\documentclass[pdflatex,sn-aps]{sn-jnl}

\usepackage{graphicx}%
\usepackage{multirow}%
\usepackage{amsmath,amssymb,amsfonts}%
\usepackage{amsthm}%
\usepackage{mathrsfs}%
\usepackage[title]{appendix}%
\usepackage{xcolor}%
\usepackage{textcomp}%
\usepackage{manyfoot}%
\usepackage{booktabs}%
\usepackage{algorithm}%
\usepackage{algorithmicx}%
\usepackage{algpseudocode}%
\usepackage{listings}%
\usepackage[UTF8, scheme=plain, punct=plain, zihao=false]{ctex}

\usepackage{xcolor}

\theoremstyle{thmstyleone}%
\newtheorem{theorem}{Theorem}[section]
\newtheorem{proposition}{Proposition}[section]%

\newtheorem{lemma}{Lemma}[section]

\theoremstyle{thmstyletwo}%
\newtheorem{remark}{Remark}[section]%

\theoremstyle{thmstylethree}%
\newtheorem{definition}{Definition}[section]%

\newcommand*{\csch}{\text{csch}}
\newcommand*{\tildeG}{{L}}

\begin{document}

\title[Generalizations of the fractional Fourier transform]{Generalizations of the fractional Fourier transform and their analytic properties}

\author*[1]{\fnm{Yue} \sur{Zhou}} 
\email{zhouy583@mail2.sysu.edu.cn}

\affil*[1]{\orgdiv{Department of Mathematics}, \orgname{Sun Yat-sen University}, \orgaddress{\city{Guangzhou}, \postcode{510275}, \state{Guangdong}, \country{P.R. China}}}

\abstract{We consider one-parameter families of quadratic-phase integral transforms which generalize the fractional Fourier transform. 
Under suitable regularity assumptions, we characterize the one-parameter groups formed by such transforms. 
Necessary and sufficient conditions for continuous dependence on the parameter are obtained in \unboldmath{$L^2$}, pointwise, and almost-everywhere senses.}



\maketitle

\section{Introduction} 
The idea of the \textit{fractional Fourier transform} (FRFT) goes back to Wiener in 1929. 
In the last few decades, the FRFT has found practical applications in 
optics \cite{FRFToptics} and signal processing \cite{FRFTsignal}, among many others. 

As a one-parameter extension of the Fourier transform 
$$\mathscr F f(u)
=\hat{f}(u):=\frac{1}{\sqrt{2\pi}}\int_\mathbb{R} e^{-iut} f(t) dt\quad (u\in\mathbb R),$$ 
the FRFT of order $\alpha\in\mathbb R$ is defined by 

\begin{equation}
\label{eq:FRFT}
    \mathscr F_\alpha f(u)=
    \sqrt{\frac{1-i\cot \alpha}{2\pi}}
    \int_{\mathbb{R}} e^{i[\frac{1}{2}(\cot\alpha) u^2-(\csc\alpha) ut+\frac{1}{2}(\cot\alpha) t^2]} f(t) dt
\end{equation}
\newpage
\noindent for $\alpha\notin\mathbb Z\pi$, 
and $\mathscr F_{n \pi} f(u)=f((-1)^n u)$
for $n\in\mathbb Z$ (\cite{FRFTNamias}, \cite{FRFTrectify}, \cite{Almeida1994}). 
Clearly, 
$$\mathscr F_{\frac\pi2} f=\mathscr F f.$$ 
It is shown in \cite{FRFTNamias} (see also \cite{FRFTrectify}, \cite{Almeida1994}) that $\mathscr F_\alpha$ satisfies Plancherel's theorem
$$\|\mathscr F_\alpha f\|_{L^2(\mathbb R)}=\|f\|_{L^2(\mathbb R)}$$
and that $\mathscr  F_\alpha f\in L^2(\mathbb R)$ depends continuously on $\alpha$ when $f\in L^2(\mathbb R)$; 
moreover, 
$\{\mathscr F_\alpha\}$ forms a \textit{one-parameter group}, 
i.e., it satisfies $\mathscr F_0={I}$ (identity) and 
\begin{equation}\label{eq:group}
\mathscr F_\alpha\circ\mathscr F_\beta=\mathscr F_{\alpha+\beta}\quad(\alpha,\beta\in\mathbb R).
\end{equation} 

More generally, 
$\{\mathscr F_\alpha\}$ is a subgroup of a three-parameter group of unitary transformations on $L^2(\mathbb R)$ called the \textit{linear canonical transforms} (LCTs; \cite{Moshinsky1971}, \cite[Ch.~9]{AlgebraicSupplement}). 
Up to suitable completion, 
the LCTs take the form 
\begin{equation}
\label{eq:LCT} 
    \mathscr L_{A,B,C}f(u)=
    D\cdot\int_{\mathbb{R}} e^{i[\frac{1}{2}A u^2-B ut+\frac{1}{2}C t^2]} f(t) dt, 
\end{equation} 
where $A, B, C\in\mathbb R$, 
$B\neq0$, 
and $D\in\mathbb C$ is an appropriate constant with 
$|D|=\sqrt{\frac{|B|}{2\pi}}$ (whose precise value will be omitted). 
For simplicity, we will denote $\mathscr L_{A,B,C}$ by the triple $[A,B,C]$. Thus, for example, 
\begin{equation}
\label{eq:F_alpha}
\mathscr F_\alpha=[\cot\alpha,\csc\alpha,\cot\alpha].
\end{equation}
Besides $\{\mathscr F_\alpha\}$, 
the LCTs contain two other one-parameter subgroups: 
\begin{align}
\mathscr E_\alpha
&:=\Big[\frac1\alpha,\frac1\alpha,\frac1\alpha\Big],\label{eq:E_alpha}\\
\mathscr G_\alpha
&:=[\coth\alpha,\csch\alpha,\coth\alpha].\label{eq:H_alpha}
\end{align}

It is easy to see that if $[A(\alpha),B(\alpha),C(\alpha)]$ ($\alpha\in\mathbb R$) is a one-parameter subgroup of the LCTs, then so are the following (where $\omega, \lambda, \gamma\in\mathbb R$ are arbitrary fixed constants): 
\begin{align}
&[A(\omega\alpha),B(\omega\alpha),C(\omega\alpha)],\label{eq:omega}\\
&\Big[\frac1\lambda A(\alpha),\frac1\lambda B(\alpha),\frac1\lambda C(\alpha)\Big],\label{eq:lambda}\\
&[A(\alpha)-\gamma,B(\alpha),C(\alpha)+\gamma].\label{eq:gamma}
\end{align}

The first objective of this note is to show that, 
up to the transformations \eqref{eq:omega}--\eqref{eq:gamma} 
and under suitable regularity assumptions on the functions $A(\alpha)$, $B(\alpha)$, and $C(\alpha)$, 
$\{\mathscr F_\alpha\}$, $\{\mathscr E_\alpha\}$, and $\{\mathscr G_\alpha\}$ give all one-parameter subgroups of the LCTs. 

\begin{theorem}
\label{thm:LCT}
Let $\mathscr L_\alpha=[A(\alpha),B(\alpha),C(\alpha)]$ $(\alpha\in\mathbb R)$ be a one-parameter subgroup of the LCTs, satisfying \eqref{eq:cond-1} and \eqref{eq:cond-2}. Then $\mathscr L_\alpha$ takes one of the following forms: 
\begin{equation}
\label{eq:thm1}
\begin{cases}
(I)\quad\Big[\frac1\lambda\cot(\omega\alpha)-\gamma,\;\frac1\lambda\csc(\omega\alpha),\;\frac1\lambda\cot(\omega\alpha)+\gamma\Big]\quad (\alpha\in\mathbb R),\\[1em]
(II)\quad\Big[\frac{1}{\lambda\alpha}-\gamma,\;\frac{1}{\lambda\alpha},\;\frac{1}{\lambda\alpha}+\gamma\Big]\quad (\alpha\in\mathbb R),\\[1em]
(III)\quad\Big[\frac1\lambda\coth(\omega\alpha)-\gamma,\;\frac1\lambda\emph{csch}(\omega\alpha),\;\frac1\lambda\coth(\omega\alpha)+\gamma\Big]\quad (\alpha\in\mathbb R), 
\end{cases}
\end{equation}
for some constants $\omega, \lambda, \gamma\in\mathbb R$. 
\end{theorem}

The one-parameter subgroups \eqref{eq:thm1} 
are essentially identified in \cite[\S 9.3]{AlgebraicSupplement} 
using Lie-theoretic methods. 
However, the problem of characterizing \textit{all} one-parameter subgroups 
of the LCTs 
is considered only formally in \cite{AlgebraicSupplement}. 
Our approach here
is more direct: 
We 
turn the subgroup property into functional equations satisfied by 
$A(\cdot)$, $B(\cdot)$, and $C(\cdot)$, 
and eventually 
reduce the problem to solving 
differential equations. 
It should be noted that 
the regularity condition of Theorem \ref{thm:LCT} 
is limited by our approach 
and can likely be weakened. 
Note also that 
there are one-parameter subgroups of the LCTs 
(corresponding to the case 
$B(\alpha)\equiv \infty$) 
which are not contained in \eqref{eq:thm1} (see \cite{AlgebraicSupplement}). 

The second objective of this note is to 
study one-parameter families 
of LCTs 
of the form 
\begin{equation}
\label{eq:gen-LCT} 
    \mathcal L_{\alpha}f(u)=
    D(\alpha)\int_{\mathbb{R}} e^{i[\frac{1}{2}A(\alpha) u^2-B(\alpha) ut+\frac{1}{2}C(\alpha) t^2]} f(t) dt
    \quad (\alpha\in I), 
\end{equation} 
where $A(\alpha), B(\alpha), C(\alpha), 
D(\alpha)$ are as in \eqref{eq:LCT}, 
and $I$ is an interval.  
More specifically, 
we are interested in finding necessary and sufficient conditions 
for $\mathcal L_{\alpha}f$ to be continuous 
in $\alpha$ 
and $f$ 
in 
certain topologies.  
To this end, 
it is useful to notice that $A(\alpha)$ 
does not influence the magnitude of 
$\mathcal L_{\alpha}f$, 
and that the integral in \eqref{eq:gen-LCT} becomes highly oscillatory 
when $|B(\alpha)|$ or $|C(\alpha)|\rightarrow\infty$. 
Note that, 
by composing $\mathscr L_{A,B,C}$ with $\mathscr F_{-\frac{\pi}{2}}$, 
one has the identity 
\begin{align}
    \mathscr L_{A,B,C}f(u)
    &=\mathscr L_{A-\frac{B^2}{C},-\frac{B}{C},-\frac{1}{C}}\hat f(u)\notag \\ 
    &=\frac{D}{\sqrt{|C|}}\,e^{i(\frac{A}{2}u^2+\frac{\pi}{4})}\int_{\mathbb{R}} e^{-i\frac{(Bu-v)^2}{2C}} \hat f(v) dv, \label{eq:LCT-FT} 
\end{align} 
which can be used to turn a highly oscillatory integral into a 
(at least formally) 
less oscillatory one. 

Assuming that $D(\alpha)$ is continuous, 
in Section \ref{sec:L2} 
we show that $\mathcal L_\alpha: L^2(\mathbb R)\rightarrow L^2(\mathbb R)$\linebreak 
is strongly continuous in $\alpha$
if and only if 
$A(\alpha)$, $B(\alpha)$, and $C(\alpha)$ 
are continuous in $\alpha$. 
Assuming that $A(\alpha)$, $B(\alpha)$, $C(\alpha)$, and $D(\alpha)$ 
are all continuous, in Section \ref{sec:pointwise}
we show that 
$\mathcal L_\alpha f(u)$ 
($u$ fixed) 
is continuous in $\alpha$ for all 
$f\in L^2(\mathbb R, (1+t^2)^{r}dt)$ 
if and only if $r>1/2$, 
and that, 
assuming additionally $C(\alpha)\neq0$, 
$\mathcal L_\alpha f(u)$ is continuous for all 
$f\in H^s(\mathbb R)$ 
(Sobolev space of order $s$) 
if and only if $s>1/2$. 

To study the case where 
$|C(\alpha)|\rightarrow\infty$, 
we will specialize to the situation where the convergence of $\mathcal L_\alpha f$ can be reduced 
by \eqref{eq:LCT-FT} to 
that of 
\begin{equation}
\label{eq:Lalpha-limit}
L_a f(u)
:=\frac{1}{\sqrt{2\pi}}
\int_{\mathbb{R}} e^{i[b(a)uv+a v^2]} \hat f(v) dv
\quad (a\rightarrow 0), 
\end{equation}
where $b(a)$ satisfies 
$\lim\limits_{a\rightarrow 0} b(a)=1$. 
This is indeed the case, 
for example, when 
$\mathcal L_\alpha=\mathscr E_\alpha$ (with $a=-\frac{\alpha}{2}$), 
$\mathscr F_\alpha$ (with $a=-\frac{\tan\alpha}{2}$), 
or $\mathscr G_\alpha$ (with $a=-\frac{\tanh\alpha}{2}$). 

Assuming that $b(a)$ is Lipschitz near $a=0$, 
in Section \ref{sec:pointwise} 
we extend a result by Carleson \cite{Carleson1980} for the case 
$b(a)=1$, by showing that 
$$\lim_{a\rightarrow0} L_af(u)=f(u),\quad \forall u\in\mathbb R$$
holds for all $f\in C^s_c(\mathbb R)$ 
(H\"older space of order $s$) 
if and only if $s>1/2$, 
and that 
$$\limsup_{a\rightarrow0} |L_a f(u)| <\infty,\quad \forall u\in\mathbb R$$
holds for all $f\in C^s_c(\mathbb R)$ 
if and only if $s\ge 1/2$. 

Under the same assumption on $b(a)$, 
in Section \ref{sec:ae-conv} 
we show that 
$$\lim_{a\rightarrow0} L_af(u)=f(u),\quad a.e.\; u\in\mathbb R$$ 
holds for all $f\in H^s(\mathbb R)$ 
if and only if $s\ge 1/4$. 
The necessity part is a direct generalization 
of a result by Dahlberg-Kenig \cite{DahlbergKenig1980} 
for the case $b(a)=1$. 
The sufficiency part 
is due to Carleson \cite{Carleson1980} 
when $b(a)=1$, 
and to Cho-Lee-Vargas \cite{Cho2012} (for $s>1/4$)
and Ding-Niu \cite{Ding2017} (for $s=1/4$) 
in the general case. 
In fact, 
\cite{Cho2012} and \cite{Ding2017} 
consider more general phase functions 
and their proofs are quite involved. 
Here we give a simplified proof utilizing the quadratic phase structure in \eqref{eq:Lalpha-limit}. 

Finally, Section \ref{sec:global-bd} 
is concerned with possible 
global bound for the maximal function  
$$L_* f(u):=\sup_{a}|L_a f(u)|,\quad u\in\mathbb R.$$ 
Extending \cite[Theorem~3]{Sjoegren2010} 
for the case $b(a)=\sqrt{1+a^2}$, 
we show that 
if $b(a)\neq 1$ and is continuous, 
then no global bound of the form 
$$L_*: H^s(\mathbb R)\rightarrow L^p(\mathbb R)$$ 
holds for $s>0$ and $p<\infty$. 
This is in sharp contrast with the case $b(a)=1$, 
for which $L_*$ 
is known to be bounded from $H^{1/4}(\mathbb R)$ to $L^4(\mathbb R)$ (see \cite{Kenig1991}). 

\section{Proof of Theorem \ref{thm:LCT}}
In this section we prove 
Theorem \ref{thm:LCT}. 
To formulate the precise conditions, we restate Theorem \ref{thm:LCT} as follows. 
\medskip

\begin{theorem}
\label{thm:LCT-2}
    Let $\mathscr L_\alpha$ $(\alpha\in\mathbb R)$ be a one-parameter family of LCTs of the form \eqref{eq:gen-LCT}, satisfying 
    \begin{align}
        &(i)\ \  A(\alpha)+C(\alpha)\ne0;\label{eq:cond-1} \\
        &(ii)\ A''(\alpha)\ \text{and}\ B''(\alpha)\ \text{exist}.\label{eq:cond-2}
    \end{align}
    Then $\mathscr L_\alpha$ $(\alpha\in\mathbb R)$ 
    forms a one-parameter group in the sense of \eqref{eq:group}
    if and only if $\mathscr L_\alpha$ belongs to one of the three classes in \eqref{eq:thm1}. 
\end{theorem}

\begin{proof}
By calculation, we get $\forall f\in L^2(\mathbb{R})$,
\begin{align*}
    \mathscr L_{\Tilde{\alpha}}\circ \mathscr L_{\alpha} f(u)=&D(\Tilde{\alpha})D(\alpha)\sqrt{\frac{i\pi}{C(\alpha)+
    A(\Tilde{\alpha})}}\times\\
    &\int_{\mathbb{R}} f(t)e^{i\big[(A(\alpha)-\frac{B^2(\alpha)}{4(A(\Tilde{\alpha})+C(\alpha))})t^2-\frac{B(\alpha)B(\Tilde{\alpha})}{2(A(\Tilde{\alpha})+C(\alpha))}ut+(C(\alpha)-\frac{B^2(\alpha)}{4(A(\Tilde{\alpha})+C(\alpha))})u^2\big]}dt.
\end{align*}
Then we get that algebraic property is equivalent to equations 
\begin{align*}
    \begin{cases}
        A(\Tilde{\alpha}+\alpha)=A(\alpha)-B^2(\alpha)\big(A(\Tilde{\alpha})+C(\alpha)\big)^{-1}&1)\\
    B(\Tilde{\alpha}+\alpha)=B(\alpha)B(\Tilde{\alpha})\big(A(\Tilde{\alpha})+C(\alpha)\big)^{-1}&2)\\
    C(\Tilde{\alpha}+\alpha)=C(\alpha)-B^2(\alpha)\big(A(\Tilde{\alpha})+C(\alpha)\big)^{-1}&3)\\
D(\Tilde{\alpha}+\alpha)=\sqrt{2\pi i}D(\alpha)D(\Tilde{\alpha})\big(A(\Tilde{\alpha})+C(\alpha)\big)^{-\frac{1}{2}}.&4)
    \end{cases}
\end{align*}

For sufficiency, we just need to substitute $(\uppercase\expandafter{\romannumeral1}), (\uppercase\expandafter{\romannumeral2})$ and $(\uppercase\expandafter{\romannumeral3})$ into the equations above. 

For necessity, firstly we show that $\forall\ \alpha,\ C(\alpha)-A(\alpha)\equiv const.$ By equation 2) and interchange $\alpha$ and $\Tilde{\alpha}$, we get
\begin{align*}
    &C(\alpha)+A(\Tilde{\alpha})=\frac{B(\Tilde{\alpha})B(\alpha)}{B(\alpha+\Tilde{\alpha})},\\
    &C(\Tilde{\alpha})+A(\alpha)=\frac{B(\Tilde{\alpha})B(\alpha)}{B(\alpha+\Tilde{\alpha})}.
\end{align*}
So $\forall\ \alpha,\Tilde{\alpha}\in\mathbb{R},\ C(\alpha)+A(\Tilde{\alpha})=C(\Tilde{\alpha})+A(\alpha)$. Then\ $\exists\ C_0\in\mathbb{R},\ C(\alpha)-A(\alpha)\equiv C_0.$ For simplicity, now let $A(\Tilde{\alpha})+C(\alpha)=(A(\Tilde{\alpha})+\frac{C_0}{2})+(A(\alpha)+\frac{C_0}{2})\triangleq a(\alpha)+a(\Tilde{\alpha}),\ B(\alpha)\triangleq\frac{1}{b(\alpha)}$. Then rewrite euqations 1)-4) below.
\begin{align*}
    \begin{cases}
        a(\Tilde{\alpha}+\alpha)-a(\alpha)=-\big(b^2(\alpha)(a(\alpha)+a(\Tilde{\alpha}))\big)^{-1}&1)'\\
        b(\alpha+\Tilde{\alpha})=b(\alpha)b(\Tilde{\alpha})(a(\alpha)+a(\Tilde{\alpha}))&2)'\\
        D(\alpha+\Tilde{\alpha})=\sqrt{2\pi i}D(\alpha)D(\Tilde{\alpha})(a(\alpha)+a(\Tilde{\alpha}))^{-\frac{1}{2}}.&3)'
    \end{cases}
\end{align*}
In order to derive the solution we want, next we need to show that $$b'(0)\ne0,\ (ab)'(0)=0.$$By 2)' we get
\begin{equation}\label{b(alpha+tilde alpha)}
    b(\alpha+\Tilde{\alpha})=(ab)(\alpha)b(\Tilde{\alpha})+(ab)(\Tilde{\alpha})b(\alpha).
\end{equation}
Differentiate w.r.t. $\Tilde{\alpha}$ and let $\Tilde{\alpha}=0$, we have 
\begin{equation}\label{b'(alpha)}
   b'(\alpha)=(ab)(\alpha)b'(0)+(ab)'(0)b(\alpha).
\end{equation}
If $b'(0)=0,\ b'(\alpha)=-2(ab)'(0)b(\alpha)$. When $(ab)'(0)=0,\ b(\alpha)\equiv const$. And then by 2)', we get $a(\alpha)\equiv const$. This leads to contradiction with 1)'. When $(ab)'(0)\ne0,\ b(\alpha)$ is an exponential function multiplied by a constant at most. Similarly by 2)', we get $a(\alpha)$ is a constant function, which also leads to contradiction with 1)'. So we get $$b'(0)\ne0.$$

For $(ab)'(0)$, we set $f(\alpha)\triangleq b(\alpha)e^{-(ab)'(0)\alpha}$. By \eqref{b(alpha+tilde alpha)} and calculation, we get
\begin{align}
    f'(\alpha)&=b'(0)(af)(\alpha),\label{Def f'(alpha)}\\
    f(\alpha+\Tilde{\alpha})&=(af)(\alpha)f(\Tilde{\alpha})+(af)(\Tilde{\alpha})f(\alpha).\label{f(alpha+tilde alpha)}
\end{align}
Differentiate w.r.t. $\Tilde{\alpha}$ in (\ref{f(alpha+tilde alpha)}) and let $\Tilde{\alpha}=0$, we have 
\begin{equation}\label{f'(alpha)}
   f'(\alpha)=(af)(\alpha)f'(0)+(af)'(0)f(\alpha).
\end{equation}
Similar to proof of $b'(0)\ne0$, we could also show that $f'(0)\ne0.$ By taking derivative of (\ref{Def f'(alpha)}) and (\ref{f'(alpha)}), we have
\begin{align}
    f''(0)&=b'(0)(af)'(0),\notag\\
    f''(0)&=2f'(0)(af)'(0).\label{f''(0)}
\end{align}
Now we suppose $f''(0)\ne0$, then $b'(0)=2f'(0)$. And because $f'(0)=b'(0)(ab)(0)$, we get $(ab)(0)=\frac{1}{2}$, and then we calculate that $$(af)'(0)=(ab)'(0)[1-(ab)(0)]=\frac{1}{2}(ab)'(0).$$ Based on these results and by (\ref{b'(alpha)}), (\ref{f'(alpha)}), we have
\begin{align*}
    \frac{1}{2}b'(\alpha)e^{-(ab)'(0)\alpha}&=\frac{1}{2}[b'(0)(ab)(\alpha)+(ab)'(0)b(\alpha)]e^{-(ab)'(0)\alpha},\\
    f'(\alpha)&=\frac{1}{2}[b'(0)(ab)(\alpha)+(ab)'(0)b(\alpha)]e^{-(ab)'(0)\alpha}.
\end{align*}
Thus, $f'(\alpha)=\frac{1}{2}b'(\alpha)e^{-(ab)'(0)\alpha}$. Since $f'(\alpha)=[b'(\alpha)-(ab)'(0)b(\alpha)]e^{-(ab)'(0)\alpha}$, we have $$b'(\alpha)=2(ab)'(0)b(\alpha).$$
However, it will lead to contradiction with equations 1)'-2)'. So $f''(0)=0$. Then by (\ref{f''(0)}), we get $(af)'(0)=0$. Now rewrite (\ref{f'(alpha)}),
\begin{equation}\label{f'(alpha) f'(0)}
    f'(\alpha)=f'(0)(af)(\alpha).
\end{equation}
When $\alpha=0$, we have that $(af)(0)=1$. Let $\alpha=\Tilde{\alpha}=0$ in (\ref{f(alpha+tilde alpha)}), then $f(0)=0$. And differentiate w.r.t. $\alpha$ in (\ref{f'(alpha) f'(0)}) and let $\alpha=0$, then $f''(0)=0$. Lastly by taking partial derivatives of second order in (\ref{f(alpha+tilde alpha)}), we get $\exists\ \omega\in\mathbb{R}$,
\begin{equation}\label{f''(alpha)}
    f''(\alpha)=\omega f(\alpha).
\end{equation}
In the following, we solve $f$ based on the value range of $\omega$.

When $\omega=0$, since $f(0)=0$, we solve that 
\begin{align*}
    f(\alpha)&=\lambda\alpha,\\
    b(\alpha)&=\lambda\alpha e^{(ab)'(0)\alpha},\ \forall\ \lambda\in\mathbb{R}
\end{align*}
Then by equations 1)'-2)', we get that $$(ab)'(0)=0.$$ 
For the same reason, when $\omega>0$ or $<0$, we could solve $f$ from (\ref{f''(alpha)}). Then by $f(0)=0$ and equations 1)'-2)' we could get that $(ab)'(0)=0$. So (\ref{b'(alpha)}) can be rewritten as
\begin{equation*}
    b'(\alpha)=b'(0)(ab)(\alpha).
\end{equation*}
This equation is similar to (\ref{f'(alpha) f'(0)}), in fact we could follow the proof above and show that $\exists\ \omega\in\mathbb{R}$,
$$b''(\alpha)=\omega b(\alpha).$$
In fact, no matter what $\omega$ is, the proof idea is same. So we just consider $\omega>0$ below.

When $\omega>0$, we solve that $$b(\alpha)=\lambda e^{\sqrt{\omega}\alpha}+\gamma e^{-\sqrt{\omega}\alpha},\ \forall\ \lambda,\omega\in\mathbb{R}.$$
Since $b(0)=f(0)=0$ and $b$ satisfies 1)'-2)', we can calculate that 
\begin{align*}
    \begin{cases}
    A(\alpha)=\frac{1}{\lambda}\coth(\omega\alpha)+\gamma\\
    B(\alpha)=\frac{1}{\lambda}\text{csch}(\omega\alpha)&\ \text{for some}\  \lambda,\omega,\gamma\in\mathbb{R}\\
    C(\alpha)=\frac{1}{\lambda}\coth(\omega\alpha)-\gamma.
    \end{cases}
\end{align*}
Now solve $D(\alpha)$, let $C_1\triangleq\sqrt{2\pi i\lambda}$. By 3)', we have $D(\alpha+\Tilde{\alpha})=C_3 D(\alpha)D(\Tilde{\alpha})\\(\coth(\omega\alpha)+\coth(\omega\Tilde{\alpha}))^{-\frac{1}{2}}.$ Take $E(\alpha)\triangleq C_1\sqrt{\sinh(\omega\alpha)}D(\alpha)$, then $\forall\ \alpha,\Tilde{\alpha}\in\mathbb{R}$, $$E(\alpha+\Tilde{\alpha})=E(\alpha)E(\Tilde{\alpha}).$$
So $\exists\ \theta\in\mathbb{R},\ E(\alpha)=e^{\theta\alpha},\ D(\alpha)=\frac{1}{C_1\sqrt{\sinh(\omega\alpha)}}e^{\theta\alpha}=\sqrt{\frac{1}{2\pi i\lambda\sinh(\omega\alpha)}}e^{\theta\alpha},\ \forall\ \theta\in\mathbb{R}$. On the other hand, we could calculate that $$\|\mathscr L_\alpha f\|_{L^2(\mathbb{R}}=\Big| D(\alpha)\sqrt{\frac{i}{2A(\alpha)}}\Big|\Big(-\frac{2A(\alpha)}{B(\alpha)}\Big)^{\frac{1}{2}}\|f\|_{L^2(\mathbb{R})}.$$
Especially when we suppose that $\mathscr L_\alpha$ is a unitary operator, we have 
$\left| D(\alpha)\sqrt{\frac{i}{2A(\alpha)}}\right|(-\frac{2A(\alpha)}{B(\alpha)})^{\frac{1}{2}}=1$. 
Then we get $e^{\theta\alpha}=\sqrt{2\pi}$. Similarly, we get the solutions $(\uppercase\expandafter{\romannumeral1})$
and $(\uppercase\expandafter{\romannumeral2}).$ 
This completes the proof of Theorem \ref{thm:LCT-2}. 
\end{proof}

\section{$L^2$ convergence}
\label{sec:L2}
In this section, 
we consider strong continuity of the operator 
$\mathcal L_\alpha: L^2(\mathbb R)\rightarrow L^2(\mathbb R)$ 
with respect to $\alpha$. 
We may assume without loss of generality that $C(\alpha)\neq0$,  so that the continuity of $\mathcal L_\alpha$ reduces 
by \eqref{eq:LCT-FT} 
to that of $G_\alpha$ 
(which we now define). 
\medskip

\begin{definition}\label{ExpreesionOnFrequency}
    $G_\alpha$ is an operator well defined on $\mathcal{S}(\mathbb{R})$. For each $\alpha,\ B(\alpha)\ne0$,
    $$G_\alpha f(u)\triangleq D(\alpha)\int_\mathbb{R}\hat{f}(\xi)e^{i[A(\alpha)\xi^2+B(\alpha)u\xi]}d\xi\ (c_0\leqslant\alpha\leqslant c_1).$$
    By density, there is a unique bounded extension of $G_\alpha$ on $L^2$.
\end{definition}
\medskip

First we apply Plancherel's theorem to $G_\alpha$. 
(For related estimates in $L^p(\mathbb{R})$, see \cite{Chen2021}.) 
\medskip 

\begin{lemma}\label{plancherel}
    For $f\in L^2(\mathbb{R})$,
    we have 
    $$\|G_\alpha f\|_{L^2(\mathbb{R})}=\left(\frac{2\pi D^2(\alpha)}{|B(\alpha)|}\right)^\frac{1}{2}\|f\|_{L^2(\mathbb{R})}.$$
\end{lemma}
\begin{proof} Take $g(\xi)\triangleq\hat{f}(\xi)e^{i\frac{1}{2}C(\alpha)\xi^2}$. By Plancherel's theorem, we get $\forall\ f\in\mathcal{S}(\mathbb{R})$,
$$\|G_\alpha f\|^2_{L^2(\mathbb{R})}=\frac{2\pi D^2(\alpha)}{|B(\alpha)|}\|g^\vee \|^2_{L^2(\mathbb{R})}=\frac{2\pi D^2(\alpha)}{|B(\alpha)|}\|f\|^2_{L^2(\mathbb{R})}.$$
By density, we have $\forall f\in L^2(\mathbb{R}),\ \|G_\alpha f\|_{L^2(\mathbb{R})}=(\frac{2\pi D^2(\alpha)}{|B(\alpha)|})^\frac{1}{2}\|f\|_{L^2(\mathbb{R})}.$
\end{proof}
\medskip

   Since we are concerned with continuous $D(\alpha)$, 
   we will assume without loss of generality that 
   $$D(\alpha)=1.$$ 

\medskip

\begin{theorem}\label{L2convergence}
    $\lim\limits_{\alpha\rightarrow\alpha_0}\|G_\alpha f-G_{\alpha_0} f\|_{L^2(\mathbb{R})}=0$ holds for all $f\in L^2(\mathbb{R})$ if and only if $A(\alpha)$ and $B(\alpha)$ are continuous at $\alpha=\alpha_0$.
\end{theorem}
\begin{proof} For $\delta>0$, define $U(\alpha_0,\delta)\triangleq\{\alpha:\ |\alpha-\alpha_0|<\delta\}$. For sufficiency, by Uniformly Bounded Theorem, it suffices to show that
\begin{align*}
    &(1)\ \exists C_0,\delta_1>0,\ \forall\alpha\in U(\alpha_0,\delta_1),\ \|G_\alpha\|\leqslant C_0.\\
    &(2)\ \forall f\in C^\infty_{0}(\mathbb{R}),\ \lim_{\alpha\rightarrow \alpha_0}\|G_\alpha f-G_{\alpha_0} f\|_{L^2(\mathbb{R})}=0.
\end{align*}
\textbf{For (1)}\quad By condition (1) and (2), we get $B(\alpha)$ exists a positive lower bound $C_0$ on a neighborhood $U(\alpha_0,\delta_1)$. Then using Lemma \eqref{plancherel}, $\forall\alpha\in U(\alpha_0,\delta_1), \ \forall f\in L^2(\mathbb{R}),\\ \|G_\alpha f\|_{L^2(\mathbb{R})}\leqslant C_0\|f\|_{L^2(\mathbb{R})}$. Then $\forall\alpha$,$\ \|G_\alpha\|\leqslant C_0.$

~\\
\textbf{For (2)}\quad It takes two steps.

\uppercase\expandafter{\romannumeral1}.\quad Show that $\forall u\in L^2(\mathbb{R}),\ \forall f\in C^\infty_{0}(\mathbb{R}),\lim\limits_{\alpha\rightarrow\alpha_0}G_\alpha f(u)=G_{\alpha_0} f(u).$ In particular, when $u$ belongs to a bounded interval $[a,b]$, the limit above is uniformly convergent w.r.t $u\in [a,b].$
\begin{align*}
    \forall u\in L^2(\mathbb{R}),\ &|G_\alpha f(u)-G_{\alpha_0} f(u)|\\
    \leqslant&\int_{\mathbb{R}} |\hat{f}(\xi)|\ |e^{i[(A(\alpha)-A(\alpha_0))\xi^2+(B(\alpha)-B(\alpha_0))u\xi]}-1|\ d\xi\\
    \leqslant&\int_{\mathbb{R}}|\hat{f}(\xi)|\ \Big[|A(\alpha)-A(\alpha_0)|\xi^2+|B(\alpha)-B(\alpha_0)|(|a|+|b|)|\xi|\Big]\ d\xi.
\end{align*}
By Dominated Convergence Theorem, the limit is zero as $\alpha\rightarrow \alpha_0.$ 

\uppercase\expandafter{\romannumeral2}.\quad Show that $\forall f\in C^\infty_{0}(\mathbb{R}),\ \exists H\in L^2(\mathbb{R}),\ \forall\alpha,\ \forall u,\ |G_\alpha f(u)|\leqslant|H(u)|.$\\
Firstly we can easily get $\exists C_1>0,\ \forall\alpha,\ |G_\alpha f(u)|\leqslant\frac{C_1}{|u|}$ by integration by parts and $|G_\alpha f(u)|\leqslant\|f\|_{L^1(\mathbb{R})}.$ Now define
\begin{equation*}
    H(u)= \begin{cases}
         \|f\|_{L^1(\mathbb{R})}&\forall\ |u|\leqslant 1,  \\
     \frac{C_1}{|u|}&\forall\ |u|>1.
    \end{cases}
\end{equation*}

By \uppercase\expandafter{\romannumeral1} and \uppercase\expandafter{\romannumeral2}, we prove that $\forall f\in C^\infty_{0}(\mathbb{R}),\ \lim\limits_{\alpha\rightarrow\alpha_0} \|G_\alpha f-G_{\alpha_0} f\|_{L^2(\mathbb{R})}=0.$

For necessity, suppose $\lim\limits_{\alpha\rightarrow\alpha_0}B(\alpha)\ne B(\alpha_0)$, then $\exists\epsilon_0>0,\ \forall\delta>0,\ \exists\alpha\in U(\alpha_0,\delta),\\ \|G_\alpha f-G_{\alpha_0}f\|_{L^2(\mathbb{R})}\geqslant\Big|\sqrt{\frac{1}{B(\alpha)}}-\sqrt{\frac{1}{B(\alpha_0)}}\Big| \sqrt{2\pi}\|f\|_{L^2(\mathbb{R})}\geqslant\epsilon_0 \sqrt{2\pi}\|f\|_{L^2(\mathbb{R})}$. It leads to contradiction with $L^2$ convergence. So $\lim\limits_{\alpha\rightarrow\alpha_0}B(\alpha)= B(\alpha_0)$ is necessary. Furthermore, let $\lim\limits_{\alpha\rightarrow\alpha_0}A(\alpha)\ne A(\alpha_0)$. Without loss of generality, suppose $A(\alpha)-A(\alpha_0)\ne0$ on $U(\alpha_0,\delta)$ and $B(\alpha)\equiv1$, then by Plancherel's theorem, we have $$\|G_\alpha f-G_{\alpha_0}f\|_{L^2(\mathbb{R})}=\|\hat{f}(\xi)(e^{i[A(\alpha)-A(\alpha_0)]\xi^2}-1)\|_{L^2(\mathbb{R})}.$$
Take $\hat{f}(\xi)=\chi_E (\xi)$, where $E=U(\sqrt{\frac{\pi}{|A(\alpha)-A(\alpha_0)}},\sqrt{\frac{1}{50|A(\alpha)-A(\alpha_0)|}})$. Then we have $\lim\limits_{\alpha\rightarrow\alpha_0}\|G_\alpha f-G_{\alpha_0}f\|_{L^2(\mathbb{R})}\ne0$. This leads to a contradiction. The proof of Theorem \ref{L2convergence} is complete. 
\end{proof}

\section{Pointwise convergence} 
\label{sec:pointwise}
In this section we study pointwise continuity of $G_\alpha f(u)$ with respect to $\alpha$. 
First we consider the problem for $f$ belonging to suitable weighted $L^2$ spaces. 
\medskip

\begin{lemma}\label{ExpressionOnTime2}
    Suppose $A(\alpha)\ne 0$. Then $$G_\alpha f(u)=\sqrt{\frac{i}{2A(\alpha)}}D(\alpha)\int_{\mathbb{R}} f(t)e^{-i\frac{(B(\alpha)u-t)^2}{4A(\alpha)}}\ dt,\quad \forall f\in C_0^\infty(\mathbb{R}).$$
\end{lemma}
\begin{proof}
 $\forall f\in C^\infty_{0}(\mathbb{R}),$
\begin{align*}
    G_\alpha f(u)&=D(\alpha)\int_{\mathbb{R}}\big(\frac{1}{\sqrt{2\pi}}\int_{\mathbb{R}} f(t)e^{-i\xi t} dt\big)\ e^{i[A(\alpha)\xi^2+B(\alpha)u\xi]}\ d\xi\\
     &=\lim\limits_{N\rightarrow +\infty}D(\alpha)\int_{-N}^{N}e^{i[A(\alpha)\xi^2+B(\alpha)u\xi]}\ \frac{1}{\sqrt{2\pi}}\int_{\mathbb{R}} f(t)e^{-i\xi t}\ dtd\xi\\
    &=\frac{D(\alpha)}{\sqrt{2\pi}}\lim\limits_{N\rightarrow +\infty}\int_{\mathbb{R}} f(t)\int_{-N}^{N}e^{i[A(\alpha)\xi^2+(B(\alpha)u-t)\xi]}\ d\xi dt\quad\text{(Fubini's Theorem).} 
\end{align*}
Since $\int_\mathbb{R}\sin x^2\ dx=\int_\mathbb{R}\cos x^2\ dx=\sqrt{\frac{\pi}{2}}$ (cf. \cite{Integralsinx2}), 
we have for $a\ne0,$ 
$$\int_\mathbb{R} e^{i(a\xi^2+b\xi)}\ d\xi=\sqrt{\frac{\pi}{a}}e^{-i\frac{b^2}{4a}+\frac{\pi}{4}i}.$$
This completes the proof of Lemma \ref{ExpressionOnTime2}.
\end{proof}
\medskip

\begin{theorem}\label{WeightedSpace}
    Suppose $A(\alpha)$ and $B(\alpha)$ are continuous at $\alpha=\alpha_0$, and $A(\alpha_0)\ne0$. 
    Then for all $f\in L^2((1+t^2)^r dt)$ with $r>\frac{1}{2}$, we have 
    \begin{equation}\label{eq:ptwise-conv}
\lim\limits_{\alpha\rightarrow\alpha_0}G_\alpha f(u)=G_{\alpha_0} f(u),\quad u\in\mathbb{R}.\end{equation}
     Conversely, if there exists $u_0\in\mathbb R$ such that $\lim\limits_{\alpha\rightarrow\alpha_0}G_\alpha f(u_0)=G_{\alpha_0} f(u_0)$ holds for all $f\in L^2((1+t^2)^r dt)$, then $r>\frac{1}{2}$. 
\end{theorem}
\begin{proof} By the uniform boundedness theorem, it suffices to show 
\begin{align*}
    \text{\uppercase\expandafter{\romannumeral1}}\quad &\forall u\in\mathbb{R},\ \forall f\in C^\infty_{0}(\mathbb{R}),\ \lim\limits_{\alpha\rightarrow\alpha_0}G_\alpha f(u)=G_{\alpha_0} f(u);\\
   \text{\uppercase\expandafter{\romannumeral2}}\quad &\forall u\in\mathbb{R},\ \exists C_0,\delta>0,\ \sup\limits_{\substack{f\in C^\infty_0(\mathbb{R})\\\alpha\in U(\alpha_0,\delta)}}\frac{|G_\alpha f(u)|}{\|f\|}\leqslant C_0,\ \text{where}\ \|f\|=\big(\int_{\mathbb{R}}|f(t)|^2 (1+t^2)^{r}dt\big)^{\frac{1}{2}}.
\end{align*}

\textbf{For \uppercase\expandafter{\romannumeral1}}\quad By condition and Theorem \ref{L2convergence}, we can get \uppercase\expandafter{\romannumeral1}.

\textbf{For \uppercase\expandafter{\romannumeral2}}\quad By Lemma \ref{ExpressionOnTime2}, When $r>\frac{1}{2}$,
\begin{align*}
    |G_\alpha f(u)|&\leqslant C_0\int_{\mathbb{R}}|f(t)|dt\\
    &\leqslant C_0\|f\|_{L^2((1+t^2)^r dt)}\big(\int_{\mathbb{R}}(1+t^2)^{-r}dt\big)^{\frac{1}{2}}\\
    &\leqslant C_0\|f\|_{L^2((1+t^2)^r dt)}.
\end{align*}
Conversely, if for all $f\in L^2((1+t^2)^r dt),\ \lim\limits_{\alpha\rightarrow\alpha_0}G_\alpha f(u_0)=G_{\alpha_0}f(u_0)$. By the uniform boundedness theorem again, $\exists C>0$, such that $\|G_\alpha\|\leqslant C$. By the Riesz representation theorem, we have $$\|G_\alpha\|=\Big\Vert\sqrt{\frac{i}{2A(\alpha)}}(1+(\cdot)^2)^{-\frac{r}{2}}e^{-i\frac{(B(\alpha)u_0-(\cdot))^2}{4A(\alpha)}}\Big\Vert_{L^2(\mathbb{R})}.$$
Thus $r>\frac{1}{2}$. 
This completes the proof of Theorem \ref{WeightedSpace}. 
\end{proof}

\begin{remark}
    When $r>\frac{1}{2}$, we have $L^2((1+t^2)^r dt)\subset  L^1(\mathbb{R})$, thus \eqref{eq:ptwise-conv} holds even when $B(\alpha_0)=0$. 
\end{remark}
\medskip

Next, we consider the pointwise convergence problem for $f$ belonging to the Sobolev spaces $H^s(\mathbb{R})$. 
\medskip

\begin{theorem}\label{sobolev}
    Suppose $A(\alpha)$ and $B(\alpha)$ are continuous at $\alpha=\alpha_0$. Then for all $f\in H^s(\mathbb{R})$ with $s>\frac{1}{2}$, we have 
    $$\lim\limits_{\alpha\rightarrow\alpha_0}G_\alpha f(u)=G_{\alpha_0} f(u),\quad  u\in\mathbb{R}.$$
    Conversely, if there exists $u_0\in\mathbb R$ such that $\lim\limits_{\alpha\rightarrow\alpha_0}G_\alpha f(u_0)=G_{\alpha_0} f(u_0)$ holds for all $f\in H^s(\mathbb{R})$, 
    then $s>\frac{1}{2}$.
\end{theorem}
\begin{proof} Similar to the proof of Theorem \ref{WeightedSpace}, it suffices to show 
\begin{align*}
    \text{\uppercase\expandafter{\romannumeral1}}\quad &\forall u\in\mathbb{R},\ \forall f\in C^\infty_{0}(\mathbb{R}),\ \lim\limits_{\alpha\rightarrow\alpha_0}G_\alpha f(u)=G_{\alpha_0} f(u);\\
    \text{\uppercase\expandafter{\romannumeral2}}\quad &\forall u\in\mathbb{R},\ \exists C_0,\delta>0,\ \sup\limits_{\substack{f\in C^\infty_0(\mathbb{R})\\\alpha\in U(\alpha_0,\delta)}}\frac{|G_\alpha f(u)|}{\|f\|}\leqslant C_0,\ \text{where}\ \|f\|=\big(\int_{\mathbb{R}}|\hat{f}(\xi)|^2 (1+|\xi|^2)^{s}dt\big)^{\frac{1}{2}}.
\end{align*}

\textbf{For \uppercase\expandafter{\romannumeral1}}\quad By Theorem \ref{L2convergence} and the assumptions, we get \uppercase\expandafter{\romannumeral1}.

\textbf{For \uppercase\expandafter{\romannumeral2}}\quad By Definition \ref{ExpreesionOnFrequency}, for $s>\frac{1}{2},$ we have 
\begin{align*}
    |G_\alpha f(u)|&\leqslant\int_\mathbb{R}|\hat{f}(\xi)|\ d\xi\\
    &\leqslant \|f\|_{H^s(\mathbb{R})}\left(\int_{\mathbb{R}}(1+\xi^2)^{-s}d\xi\right)^{\frac{1}{2}}\\
    &\leqslant C_0\|f\|_{H^s (\mathbb{R})}.
\end{align*}
Conversely, if for all $f\in H^s(\mathbb{R}),\ \lim\limits_{\alpha\rightarrow\alpha_0}G_\alpha f(u_0)=G_{\alpha_0} f(u_0)$ holds. Then $\exists C>0$ such that $\|G_\alpha\|\leqslant C$. By the Riesz theorem again, we have 
$$\|G_\alpha\|=\Big\Vert(1+(\cdot)^2)^{-\frac{s}{2}}e^{i[(A(\alpha)(\cdot))^2+B(\alpha)u_0(\cdot)]}\Big\Vert_{L^2(\mathbb{R})}.$$
Thus $s>\frac{1}{2}$. 
This completes the proof of 
Theorem \ref{sobolev}
\end{proof}

We turn to the operator $L_a$ defined in \eqref{eq:Lalpha-limit} and consider its limit as $a\rightarrow0$. The space of H\"older continuous functions supported in $(c,d)$ will be denoted by $C^s_c (c,d)$, 
where $s$ is the H\"older exponent. 
Since 
    $C_c^s (c,d)\subset  H^s(\mathbb{R})$ (cf. \cite[\S1.4]{grafakosmodern}), 
    by Theorem \ref{sobolev} the pointwise convergence \eqref{eq:ptwise-conv} holds for all $f\in C_c ^s (c,d)$ with $s>1/2$. 

\medskip 

We will show that the condition $s>\frac{1}{2}$ is also optimal for H\"older continuous functions. 
First we state a positive result. 
\medskip 

\begin{theorem}\label{c1/2bounded}
    Suppose $b(a)$ is Lipschitz continuous on $[0,c_1]$, and $b(0)=1$. 
    Then there exists a constant $C_2>0$, such that 
    for all $R>|c|+|d|$, we have 
    $$\Big\Vert\sup\limits_{a\in[0,c_1]} |\tildeG_a f|\Big\Vert_{L^\infty[-R,R]}\leqslant C_2(1+|d-c|)R^\frac{1}{2} \|f\|_{C^{\frac{1}{2}}_c (c,d)}.$$
\end{theorem}
\begin{proof}
For simplicity, denote $b(a)\triangleq b$ and $\|f\|_{C_c ^{\frac{1}{2}} (c,d)}\triangleq\|f\|$. By Lemma \ref{ExpressionOnTime2},
\begin{align*}
    \big|\tildeG_a f(u)\big|&=\big|\frac{1}{\sqrt{2a}}\int_{\mathbb{R}} f(t)e^{-i\frac{(b(a)u-t)^2}{4a}}\ dt\big|\\
    &=\big|\frac{1}{\sqrt{2a}}\int_{c-b(a)u}^{d-b(a)u}f(t+bu)e^{-i\frac{t^2}{4a}}\ dt|\\
    &\leqslant |2a|^{-\frac{1}{2}}\big|\int_{c-bu}^{d-bu}f(t+u)e^{-i\frac{t^2}{4a}}\ dt\big|\\
    &+|2a|^{-\frac{1}{2}}\big|\int_{c-bu}^{d-bu}\big(f(t+bu)-f(t+u)\big)e^{-i\frac{t^2}{4a}}\ dt\big|\\
    &\triangleq \uppercase\expandafter{\romannumeral1}+\uppercase\expandafter{\romannumeral2}. 
\end{align*}

When $u\ne 0$, without loss of generality, suppose the lower limit of the integral above satisfies $c-bu\geqslant0\ ($if $c-bu<0$, we can divide $\uppercase\expandafter{\romannumeral1}$ into $\int_{c-bu}^{0}$ and $\int_{0}^{d-bu})$. Now we divide $\uppercase\expandafter{\romannumeral1}$ into several parts. Let $t_k=\sqrt{8k\pi a}\ (k=n_1+1,n_1+2,\cdots,n_2;\ n_1,n_2>0),\ t_{n_1}=c-bu\in[\sqrt{8n_1\pi a},\sqrt{8(n_1+1)\pi a}),\ t_{n_2+1}=d-bu\in(\sqrt{8n_2\pi a},\sqrt{8(n_2+1)\pi a}]$.
\begin{align*}
    \uppercase\expandafter{\romannumeral1}&\leqslant|2a|^{-\frac{1}{2}}\sum_{k=n_1}^{n_2} |\int_{t_k}^{t_{k+1}}(f(t+u)-f(t_k+u))e^{-i\frac{t^2}{4a}}\ dt\big|\\
    &+|2a|^{-\frac{1}{2}}\sum_{k=n_1}^{n_2}\big|f(t_k+u)\int_{t_k}^{t_{k+1}}e^{-i\frac{t^2}{4a}}\ dt\big|.\\
    &\triangleq\uppercase\expandafter{\romannumeral1}^{(1)}+\uppercase\expandafter{\romannumeral1}^{(2)}.\\
    \uppercase\expandafter{\romannumeral1}^{(1)}&\lesssim |a|^{-\frac{1}{2}}\|f\|\sum_{k=n_1}^{n_2}|t_{k+1}-t_k|^{\frac{3}{2}}\lesssim|a|^{\frac{1}{4}}\|f\|\sum_{k=n_1+1}^{n_2+1}k^{-\frac{3}{4}}\lesssim R^{\frac{1}{2}}\|f\|.\\ 
    \uppercase\expandafter{\romannumeral1}^{(2)}&\lesssim\|f\| \sum_{k=n_1}^{n_2}\Big|\int_{\frac{t_{k+1} ^2}{4a}}^{\frac{t_{k} ^2}{4a}}e^{-it}\frac{1}{\sqrt{t}}\ dt \Big|. 
\end{align*}

Now divide $\sum\limits_{k=n_1}^{n_2}\Big|\int_{\frac{t_{k+1} ^2}{4a}}^{\frac{t_{k} ^2}{4a}}e^{-it}\frac{1}{\sqrt{t}}\ dt \Big|$ into three parts:\\ 
(1)\ Using integration by parts, 
$$\sum\limits_{k=n_1+1}^{n_2-1}\Big|\int_{2k\pi}^{2(k+1)\pi}e^{-it}\frac{1}{\sqrt{t}}\ dt \Big|\lesssim\sum\limits_{k=n_1+1}^{n_2-1}(\frac{1}{\sqrt{k}}-\frac{1}{\sqrt{k+1}})\leqslant(n_1+1)^{-\frac{1}{2}}\leqslant1;$$\\
(2)\ $\Big|\int_{\frac{(c-bu)^2}{4a}}^{2(n_1+1)\pi}e^{-it}\frac{1}{\sqrt{t}}\ dt\Big|$: When $n_1\leqslant\ 3,\ \Big|\int_{\frac{(c-bu)^2}{4a}}^{2(n_1+1)\pi}\Big|\leqslant\int_0^{20}t^{-\frac{1}{2}}dt$; Otherwise, \\$\Big|\int_{\frac{(c-bu)^2}{4a}}^{2(n_1+1)\pi}\Big|\lesssim\int_{2n_1\pi}^{2(n_1+1)\pi}t^{-\frac{1}{2}}dt\lesssim n_1 ^{-\frac{1}{2}}\leqslant\frac{1}{\sqrt{3}}$;\\
(3)\ $\Big|\int_{2n_2 \pi a}^{\frac{(d-bu)^2}{4a}}e^{-it}\frac{1}{\sqrt{t}}\ dt\Big|$: Similar to (2). 
~\\
Thus, by the Lipschitz continuity of $b(a)$, 
we have 
$$\ \uppercase\expandafter{\romannumeral2}\lesssim |a|^{-\frac{1}{2}}|(d-c)^2(b-1)u|^{\frac{1}{2}}\|f\|\lesssim |d-c|R^{\frac{1}{2}}\|f\|. $$
Since $f\in C^{1/2}_c(c,d)$, it is easy to verify that 
$\tildeG_a f(u)$ is continuous with respect to $u$. 
Therefore the case $u=0$ follows immediately from the bound for $u\neq0$. 
This completes the proof of Theorem \ref{c1/2bounded}. 
\end{proof}
\medskip 

Now we state the negative result. 
\medskip

\begin{proposition}\label{c1/2counterexample}
    Suppose $|b(a)|\leqslant M$ and $b(0)=1$. Then $\exists\  f\in C^{\frac{1}{2}}_c [\frac{1}{4},\frac{1}{2}]$ and $u_0\in (0,\frac{1}{80M})$, such that 
    $$\lim\limits_{a\rightarrow0}\tildeG_a f(u_0)\ne\tildeG_0 f(u_0).$$
\end{proposition}
\begin{proof}
Let $a_k=2^{-k},\ b_k=b(a_k)$. Take $f(t)=\sum\limits_{k=1}^{+\infty}\sqrt{2a_k}\ e^{i\frac{(b_k u_0-t)^2}{4a_k}}\phi(t)$, where $\phi\in C^\infty_{0}(\mathbb{R})$ and $supp\ \phi\subset [\frac{1}{4},\frac{1}{2}]$. It is easy to see that $f\in C[0,1]$. Let $0<u_0<\frac{1}{80M}$. Then  
\begin{equation*}
    |\tildeG_{a_n}f(u_0)|=\Big|\sum\limits_{k=1}^{+\infty}\sqrt{\frac{a_k}{a_n}}\int_{\mathbb{R}}\phi(t)\ e^{i[\frac{(b_k u_0-t)^2}{4a_k}-\frac{(b_n u_0-t)^2}{4a_n}]}\ dt\Big|.
\end{equation*}
Now divide $\tildeG_{a_n}f(u_0)$ into three parts.

~\\
(1)\ $k=n:\quad \int_{\mathbb{R}}\phi(t)\ dt\ne0.$\\
(2)\ $k<n:$ 
\begin{align*}
    &\Big|\sum\limits_{k=1}^{n-1}2^{\frac{n-k}{2}}\int_{\mathbb{R}}\phi(t)e^{i[2^{k-2}(b_k u_0-t)^2-2^{n-2}(b_n u_0-t)^2]}\ dt\Big|\\
    \lesssim&\sum\limits_{k=1}^{n-1}\int_{\mathbb{R}}\Big|\frac{\phi'(t)}{2^{k-1}(b_k u_0-t)-2^{n-1}(b_n u_0-t)}\Big|\ dt\\
    +&\int_{\mathbb{R}}\frac{|\phi'(t)|\ |2^{k-1}-2^{n-1}|}{|2^{k-1}(b_k u_0-t)-2^{n-1}(b_n u_0-t)|^2}\ dt\quad\text{(using integration by parts)} \\
    \lesssim&\sum\limits_{k=1}^{n-1} \frac{2^{\frac{n-k}{2}}}{|2^{k-1}-2^{n-1}|}\\
    \lesssim&\sum\limits_{k=1}^{n-1} 2^{\frac{n-k}{2}}2^{-n}\\
    \lesssim&2^{-\frac{n}{2}}\rightarrow0,\quad \text{as}\ n\rightarrow+\infty.
\end{align*}
(3)\ $k>n:$ 
\begin{align*}
    &\Big|\sum\limits_{k=n+1}^{+\infty}2^{\frac{n-k}{2}}\int_{\mathbb{R}}\phi(t)e^{i[2^{k-2}(b_k u_0-t)^2-2^{n-2}(b_n u_0-t)^2]}\ dt\Big|\\
    \lesssim&\sum\limits_{k=n+1}^{+\infty} \frac{2^{\frac{n-k}{2}}}{|2^{k-1}-2^{n-1}|}\\
    \lesssim&\sum\limits_{k=n+1}^{+\infty}2^{\frac{n-k}{2}}2^{-k}\\
    \lesssim&\frac{1}{2^n}\sum\limits_{l=1}^{+\infty}2^{-\frac{l}{2}}\rightarrow0,\quad \text{as}\ n\rightarrow+\infty.
\end{align*}
Since $\tildeG_0 f(u_0)=0$, we get $\lim\limits_{n\rightarrow+\infty}\tildeG_{a_n}f(u_0)\ne\tildeG_0 f(u_0)$. 
It remains to show $f\in C^{\frac{1}{2}}_c [\frac{1}{4},\frac{1}{2}]$, that is, $\exists C>0$, such that $\forall t_1,\ t_2\in [\frac{1}{4},\frac{1}{2}],$ we have 
$$|f(t_1)-f(t_2)|\leqslant C|t_1-t_2|^{\frac{1}{2}}.$$
Since $t_1\ne t_2,\ \exists N>0,$ such that $2^{-(N+1)}<|t_1-t_2|\leqslant2^{-N}.$ Now 
\begin{align*}
    |f(t_1)-f(t_2)|\lesssim&\sum\limits_{k=1}^{+\infty}|\sqrt{2a_k}|\ \big|e^{i\frac{(b_k u_0-t_1)^2}{4a_k}}\phi(t_1)-e^{i\frac{(b_k u_0-t_2)^2}{4a_k}}\phi(t_2)\big|\\
    =&\big(\sum\limits_{k\leqslant N}+\sum\limits_{k>N}\big)|\sqrt{2a_k}|\ \big|e^{i\frac{(b_k u_0-t_1)^2}{4a_k}}\phi(t_1)-e^{i\frac{(b_k u_0-t_2)^2}{4a_k}}\phi(t_2)\big|\\
    \triangleq&\uppercase\expandafter{\romannumeral1}+\uppercase\expandafter{\romannumeral2},
    \end{align*}
    where 
    \begin{align*}    \uppercase\expandafter{\romannumeral1}&\lesssim\sum\limits_{k\leqslant N}2^{-\frac{k}{2}}(2^k+1)\ |t_1-t_2|\qquad\text{(by the mean value theorem)}\\
    &\lesssim2^{\frac{N}{2}}|t_1-t_2|\lesssim|t_1-t_2|^{\frac{1}{2}},\\
    \uppercase\expandafter{\romannumeral2}&\lesssim\sum\limits_{k>N}2^{-\frac{k}{2}}\\
    &\lesssim2^{-\frac{N}{2}}\lesssim|t_1-t_2|^{\frac{1}{2}}.
\end{align*}
Thus $f\in C^{\frac{1}{2}}_c [\frac{1}{4},\frac{1}{2}].$ 
This completes the proof of Proposition \ref{c1/2counterexample}. 
\end{proof}

\section{Almost-everywhere convergence}
\label{sec:ae-conv}
In this section we consider the almost-everywhere convergence of $L_af$ as $a\rightarrow a_0$. 
\medskip 

The following result is a special case of \cite[Theorem 1.3]{Cho2012} and \cite[Theorem 1.3]{Ding2017}. Here we give a simplified proof for the specific operator $L_a$. 
(For related results in higher dimensions, we refer to 
\cite{Du2017}, \cite{Du2019}, \cite{Li2021}, and references therein.) 
\medskip 

\begin{proposition}\label{a.e.}
    Suppose $b(a)$ is Lipschitz continuous near $a=a_0$. Then for all 
$f\in H^s (\mathbb{R})$ with $s\geqslant\frac{1}{4}$, we have $$\lim\limits_{a\rightarrow a_0}\tildeG_a f(u)=\tildeG_{a_0} f(u),\quad a.e. \ u\in\mathbb R.\footnote{For more quantitative results concerning the rate of convergence, see \cite{Cao2018}.}$$
\end{proposition}

Before proving Proposition \ref{a.e.}, we first recall that the almost-everywhere convergence of $\tildeG_a f$ is implied by boundedness of the corresponding maximal operators 
$$\tildeG_*^{+} f(u)\triangleq\sup\limits_{a\in (a_0,a_0+\delta)} |\tildeG_a f(u)|$$
and 
$$\tildeG_{*}^- f(u)\triangleq\sup\limits_{a\in (a_0-\delta,a_0)} |\tildeG_a f(u)|.$$ 
For simplicity, we will only consider $\tildeG_* f\triangleq\tildeG_{*}^+ f$. 
\medskip

\begin{lemma}\label{lem:1}
    Let $E\subset\mathbb R$. Suppose 
    \begin{align*}
        &(1)\lim\limits_{a\rightarrow a_0} b(a)=b(a_0);\\
        &(2)\ \exists\ C>0,\ \text{such that }  
        \|\tildeG_* f\|_{L^{1,\infty}(E)}\leqslant C\|f\|_{H^s (\mathbb{R})}.
        \end{align*}
Then for all $f\in H^s (\mathbb{R})$, we have  $$\lim\limits_{a\rightarrow a_0}\tildeG_a f(u)=\tildeG_{a_0} f( u),\quad a.e.\ u\in E.$$
    
\end{lemma}
\begin{proof}
The proof follows from a standard density argument, cf. \cite[\S 1.1]{stein1970singular}. 
\end{proof}
\medskip

To bound the maximal function $\tildeG_* f$, 
we use the Kolmogorov-Seliverstov-Plessner method (cf. \cite{Carleson1980}). 
\medskip

\begin{lemma}\label{euqivlemma}
    Let $E\subset\mathbb R$. Then 
    $$\exists C>0,\ \|\tildeG_* f\|_{L^{1,\infty} (E)}\leqslant C\|f\|_{H^s (\mathbb{R})}\Leftrightarrow\exists\tilde{C}>0,\ \sup\limits_{\{a(\cdot)\}}\|\tildeG_{\{a(\cdot)\}} f\|_{L^{1,\infty} (E)}\leqslant\tilde{C}\|f\|_{H^s (\mathbb{R})},$$
where $a(u)$ is an arbitrary function whose value lies in $(a_0,a_0+\delta)$. 
\end{lemma} 
\medskip

As in \cite{Carleson1980} 
we will need the following integral estimate.
\medskip

\begin{lemma}\label{LemmaIntegral}
    There exists a constant $C_0>0$ independent of $a,b\in\mathbb R$, such that for any $N>\max(\frac{1}{\sqrt{|b|}},\frac{1}{|b|})$, we have 
    $$\left|\int_{-N}^{N}(1+\xi^2)^{-\frac{1}{4}}e^{i(a\xi^2+b\xi)}\ d\xi\right|\leqslant C_0(|a|^2+|b|^2)^{-\frac{1}{4}}.$$
\end{lemma}
\begin{proof}
By \cite[Ch. 8]{stein1993harmonic}, when $b=0$, $\exists C_1>0$, such that for $N>0$, we have 
$$|\int_{-N}^{N}(1+\xi^2)^{-\frac{1}{4}}e^{ia\xi^2}\ d\xi|\leqslant C_1|a|^{-\frac{1}{2}};$$ 
when $b\ne0$, $\exists C_2>0$, such that for $N>\max(\frac{1}{\sqrt{|b|}},\frac{1}{|b|})$, we have 
$$|\int_{-N}^{N}(1+\xi^2)^{-\frac{1}{4}}e^{i(a\xi^2+b\xi)}\ d\xi|\leqslant C_2|b|^{-\frac{1}{2}}.$$ 
Combining the above estimates, 
the desired bound follows. 
\end{proof}

~\\
We are now ready to prove Proposition \ref{a.e.}. 

\begin{proof}[Proof of Proposition \ref{a.e.}]
By Lemma \ref{lem:1} and Lemma \ref{euqivlemma}, it suffices to show that there exists $p\geqslant1$, such that for any bounded interval $[c,d]\subset \mathbb{R},\ \exists C>0$, such that 
\begin{equation}
    \sup\limits_{\{a(\cdot)\}}\|\tildeG_{a(\cdot)} f\|_{L^p [c,d]}\leqslant C\|f\|_{H^{\frac{1}{4}} (\mathbb{R})}.\label{Localboundedness}
\end{equation}
$\forall f\in C^\infty_{0}(\mathbb{R})$, we have $\|\tildeG_{a(u)}f\|_{L^p[c,d]}=\sup\limits_{\substack{g\in C^\infty_{0}(\mathbb{R})\\\|g\|_{L^{p'}[c,d]}\leqslant1}}\langle\tildeG_{a(u)}f,g\rangle$, where $\frac{1}{p}+\frac{1}{p'}=1$. By Fubini's theorem,
\begin{align*}
    \langle\tildeG_{a(u)}f,g\rangle=&\int_c^d\int_{\mathbb{R}}\hat{f}(\xi)e^{i(a(u)\xi^2+b(a(u))u\xi)}\ d\xi\ g(u)\ du\\
    =&\lim\limits_{N\rightarrow+\infty}\int_{-N}^N\Big[(1+\xi^2)^{\frac{s}{2}}\hat{f}(\xi)\Big](1+\xi^2)^{-\frac{s}{2}}\big[\int_c^d e^{i(a(u)\xi^2+b(a(u))u\xi)} g(u)\ du\big]\ d\xi\\
    =&\lim\limits_{N\rightarrow+\infty}\int_{-N}^N(1+\xi^2)^{\frac{s}{2}}\hat{f}(\xi)\times\\
    \Big[&(1+\xi^2)^{-s}\int_c^d\int_c^d g(u_1)\overline{g(u_2)}e^{i[(a(u_1)-a(u_2))\xi^2+(b(a(u_1))u_1-b(a(u_2))u_2)\xi]}\ du_1 du_2\Big]^{\frac{1}{2}}d\xi\\
    \leqslant&\|f\|_{H^{\frac{1}{4}} (\mathbb{R})}\sup\limits_{N}\Big\Vert\Big[(1+(\cdot)^2)^{-s}\int_c^d\int_c^d g(u_1)\overline{g(u_2)}e^{i[(a_1-a_2)(\cdot)^2+(b_1 u_1-b_2 u_2)(\cdot)]}\ du_1 du_2\Big]^{\frac{1}{2}}\Big\Vert_{L^2[-N,N]},
\end{align*}
where for simplicity, let $a_i\triangleq a(u_i),\ b_i\triangleq b(a(u_i))$. By Fubini's theorem, 
\begin{align*}
    &\Big\Vert\Big[(1+(\cdot)^2)^{-s}\int_c^d\int_c^d g(u_1)\overline{g(u_2)}e^{i[(a_1-a_2)(\cdot)^2+(b_1 u_1-b_2 u_2)(\cdot)]}\ du_1 du_2\Big]^{\frac{1}{2}}\Big\Vert^2_{L^2[-N,N]}\\
    =&\int_c^d\int_c^d g(u_1)\overline{g(u_2)}\int_{-N}^{N}(1+\xi^2)^{-s} e^{i[(a_1-a_2)\xi^2+(b_1 u_1-b_2 u_2)\xi]}\ d\xi du_1 du_2
\end{align*}

Now estimate $K(u_1,u_2)\triangleq\int_{-N}^{N}(1+\xi^2)^{-s} e^{i[(a_1-a_2)\xi^2+(b_1 u_1-b_2 u_2)\xi]}\ d\xi$. Divide $[c,d]^2$ into several parts:\\
\textbf{\uppercase\expandafter{\romannumeral1}}\quad When $(u_1,u_2)\in\{[c,d]^2:a_1=a_2,\ b_1 u_1\ne b_2 u_2\}$, by Lemma \ref{LemmaIntegral},\ $|K(u_1,u_2)|\leqslant C_0|b_1 u_1-b_2 u_2|^{-\frac{1}{2}}$. Since\ $a_1=a_2\Rightarrow\ b_1=b_2$, $$|b_1 u_1-b_2 u_2|^{-\frac{1}{2}}=|b_1|^{-\frac{1}{2}} |u_1-u_2|^{-\frac{1}{2}}\leqslant C|u_1-u_2|^{-\frac{1}{2}}.$$
\textbf{\uppercase\expandafter{\romannumeral2}}\quad When $(u_1,u_2)\in\{[c,d]^2:a_1\ne a_2,\ b_1 u_1=b_2 u_2\}$, by Lemma \ref{LemmaIntegral},\ $|K(u_1,u_2)|\leqslant C_0|a_1-a_2|^{-\frac{1}{2}}$. Since $b$ is Lipschitz continuous,$$|a_1-a_2|\geqslant C|b_1-b_2|=C|\frac{b_1}{u_2}u_2-\frac{b_2}{u_1}u_1|=C|\frac{b_1}{u_2}(u_1-u_2)|\geqslant C|u_1-u_2|.$$
\textbf{\uppercase\expandafter{\romannumeral3}}\quad When $(u_1,u_2)\in\{[c,d]^2:a_1\ne a_2,\ b_1 u_1\ne b_2 u_2\}$, for simplicity, let $b_1 u_1\triangleq\frac{1}{\epsilon_1},\ b_2 u_2\triangleq\frac{1}{\epsilon_2}$. Since $|a_1-a_2|\geqslant C|b_1-b_2|=C|\frac{1}{\epsilon_1 u_1}-\frac{1}{\epsilon_2 u_2}|$, by Lemma \ref{LemmaIntegral}, we get
$$|K(u_1,u_2)|\leqslant C_0\Big(|\frac{1}{\epsilon_1}-\frac{1}{\epsilon_2}|+|\frac{1}{\epsilon_1 u_1}-\frac{1}{\epsilon_2 u_2}|\Big)^{-\frac{1}{2}}.$$
Since $\Big(|\frac{1}{\epsilon_1}-\frac{1}{\epsilon_2}|+|\frac{1}{\epsilon_1 u_1}-\frac{1}{\epsilon_2 u_2}|\Big)^2 \leqslant C_0\epsilon^2 \Big[u_1^2 u_2^2\ (\frac{\epsilon_2}{\epsilon_1}-1)^2+(\frac{\epsilon_2 u_2}{\epsilon_1}-u_1)^2\Big]$, now consider $u_1^2 u_2^2\ (\frac{\epsilon_2}{\epsilon_1}-1)^2+(\frac{\epsilon_2 u_2}{\epsilon_1}-u_1)^2$ as a function w.r.t. $\frac{\epsilon_2}{\epsilon_1}$, and we have $$\Big(|\frac{1}{\epsilon_1}-\frac{1}{\epsilon_2}|+|\frac{1}{\epsilon_1 u_1}-\frac{1}{\epsilon_2 u_2}|\Big)^2\geqslant C_0 (\epsilon_1 u_1)^2 (u_1-u_2)^2\geqslant C_0(u_1-u_2)^2.$$
\textbf{\uppercase\expandafter{\romannumeral4}}\quad When $(u_1,u_2)\in\{[c,d]^2:a_1=a_2,\ b_1 u_1=b_2 u_2\}$, by \textbf{\uppercase\expandafter{\romannumeral3}}, we know that $$(|a_1-a_2|+|b_1 u_1-b_2 u_2|)^2\geqslant C_0(u_1-u_2)^2.$$
So $\{[c,d]^2:a_1=a_2,\ b_1 u_1=b_2 u_2\}\subset  \{[c,d]^2:u_1=u_2\}$ be a zero measure. Without loss of generality, we can suppose that $a_1=a_2$ and $b_1 u_1=b_2 u_2$ do not hold at the same time.

To sum up, we have $\exists N_0>0,\ \forall N>N_0,$
\begin{align*}
    &\Big|\int_c^d\int_c^d g(u_1)\overline{g(u_2)}K(u_1,u_2)\ du_1 du_2\Big|\\
    \leqslant& C\int_c^d\int_c^d |g(u_1)g(u_2)|\ |u_1-u_2|^{-\frac{1}{2}}\ du_1 du_2\\
    \leqslant& C\langle H*|g|,|g|\rangle\qquad(H(u_1)\triangleq|u_1|^{-\frac{1}{2}})\\
    \leqslant& C\|H*|g|\|_{L^{p}[c,d]}\|g\|_{L^{p'}[c,d]}\\
    \leqslant&C\|H\|_{L^{\frac{p}{2},\infty}[c,d]} \|g\|^2_{L^{p'}[c,d]}\\
    \leqslant&C\|H\|_{L^{\frac{p}{2},\infty}[c,d]},\ \text{where}\ C\ \text{is independent of}\ N.
\end{align*}
In the second last line, we have used Young's inequality on weak type spaces (cf. \cite[Ch. 1]{grafakos2008classical}). Now all we need is $\|H\|_{L^{\frac{p}{2}},\infty}<\infty$, that is, $p\leqslant4$. It follows that \eqref{Localboundedness} holds for $p\le 4$. 
This completes the proof of Proposition \ref{a.e.}. 
\end{proof}

\begin{remark}
    The condition $p\le 4$ is optimal for \eqref{Localboundedness} to hold. This follows from the sharpness of the Sobolev inequality: 
    $$\|f\|_{L^p [c,d]}=c\|L_0 f\|_{L^p [c,d]}    \leqslant c\|\tildeG_* f\|_{L^p [c,d]}\leqslant C\|f\|_{H^{\frac{1}{4}} (\mathbb{R})},\quad p\le 4.$$
\end{remark}

Now we show that the almost-everywhere convergence in Proposition \ref{a.e.} does not always hold if $s<\frac{1}{4}$. 
As in \cite[p. 208]{DahlbergKenig1980}, the proof reduces to the following proposition. 
\medskip

\begin{proposition}\label{necessitya.e.}
    Suppose $b(a)$ is continuous on $[a_0,a_0+\delta)$. If there exists a set $E$ of positive measure such that 
    $$\|\tildeG_* f\|_{L^{1,\infty}(E)} \leqslant C\|f\|_{H^s (\mathbb{R})},$$ 
    then $s\geqslant\frac{1}{4}.$
\end{proposition}
\begin{proof}
For any $N>10$, take $\Phi\in C^\infty_{0}(\mathbb{R}),\ \Phi_N\triangleq N\Phi(Nx)$ such that $|\tildeG_{a_0}\Phi_N(x)|>\frac{N}{2}$ on $[-\frac{1}{N},\frac{1}{N}]$. By Theorem \ref{L2convergence}, we have $\lim\limits_{a\rightarrow a_0}|\tildeG_a \Phi_N (u)|=|\tildeG_{a_0}\Phi_N(u)|$ uniformly w.r.t. $u\in [-\frac{1}{N},\frac{1}{N}],\ i.e.\ \exists\epsilon_0<\frac{1}{50},\ \exists\delta>0,\ \forall a\in U(a_0,\delta),\ \forall u\in [-\frac{1}{N},\frac{1}{N}],\ |\tildeG_a \Phi_N (u)|>|\tildeG_{a_0}\Phi_N(u)|-\epsilon_0>\frac{N}{4}$. Since by calculation, $$\tildeG_a(e^{i\frac{a-a_0}{a}Ny} \Phi_N(y))(u)=e^{i\big(\frac{[(a-a_0)N]^2}{a}+\frac{b(a)u(a-a_0)}{a}N\big)}\tildeG_a \Phi_N(u+\frac{2(a-a_0)N}{b(a)}).$$
Then $\forall u\in E,\ \exists a=a(u)\in U(a_0,\delta),\ s.t.\ u=\frac{2(a-a_0)N}{b(a)},$ $$|\tildeG_a(e^{-i\frac{a-a_0}{a}Ny}\Phi_N(y))(u)|>|\tildeG_{a_0}\Phi_N(u-\frac{2(a-a_0)N}{b(a)})|-\epsilon_0>\frac{N}{4}.$$
This is because the range of $\frac{2(a-a_0)}{b(a)}N$ can be denoted by $[C_1 N,C_2 N]\ (C_1<C_2)$ by continuity of $b$, and we can choose $N$ at the beginning such that $E\subset [C_1 N,C_2 N]$. Now we have $\|\tildeG_* (e^{-i\frac{a-a_0}{a}N y}\Phi_N(y))\|_{L^{1,\infty}(E)}\geqslant\frac{N}{4}$. Since $\|e^{-i\frac{a-a_0}{a}N (\cdot)}\Phi_N(\cdot)\|_{H^s (\mathbb{R})}\leqslant C_0 N^{2s+\frac{1}{2}}$,  we get $$\frac{N}{4}\leqslant \|\tildeG_* (e^{-i\frac{a-a_0}{a}N y}\Phi_N(y))\|_{L^{1,\infty}(E)}\leqslant C\|e^{-i\frac{a-a_0}{a}N (\cdot)}\Phi_N(\cdot)\|_{H^s (\mathbb{R})}\leqslant CN^{2s+\frac{1}{2}}.$$ So $1\leqslant 2s+\frac{1}{2}\Rightarrow s\geqslant\frac{1}{4}.$
This completes the proof of Proposition \ref{necessitya.e.}.  
\end{proof}

\section{Global boundedness}
\label{sec:global-bd}
In this section we consider global boundedness of the maximal function $L_*f$. 
\medskip

\begin{proposition}
\label{prop:global}
    Suppose the range of $b(a)$,  $a\in(a_0,a_0+\delta)$ is a  set of positive measure. Then $$ \sup\limits_{\substack{f\in H^s (\mathbb{R})\\ f\ne0}}\frac{\|\tildeG_* f\|_{L^p(\mathbb{R})}}{\|f\|_{H^s (\mathbb{R})}} =+\infty$$
    for any $s>0$ and $p<\infty$. 
\end{proposition}
\begin{proof}
The proof is an adaptation of that of \cite[Theorem 3]{Sjoegren2010}. 
By Lemma \ref{euqivlemma}, it suffices to show that $\forall L>0,\ \exists f_0\in H^s (\mathbb{R}),\ \exists a(u),\ \|\tildeG_{a(\cdot)} f_0\|_{L^2(\mathbb{R})} >L\|f_0\|_{H^s (\mathbb{R})}$. Without loss of generality, suppose $a_0+\delta\ne0$. Take $\hat{f}_0(\xi)=e^{iN\xi}\chi_{U(0,\frac{1}{\sqrt{a_0+\delta}})}(\xi)$. $\forall u\in F\triangleq\{-\frac{N}{x}:\ \forall x\in E\},\ \exists\ \tilde{a}(u)\in(a_0,a_0+\delta), s.t.\ b(\tilde{a}(u))u+N=0.$
\begin{align*}
    |\tildeG_{\tilde{a}(u)}f_0(u)|&=|\int_{-\frac{1}{\sqrt{a_0+\delta}}}^{\frac{1}{\sqrt{a_0+\delta}}} e^{i\tilde{a}(u)\xi^2}\ d\xi|\\
    &>|\frac{2}{\sqrt{\tilde{a}(u)}}\int_{0}^{\sqrt{\frac{\tilde{a}(u)}{a_0+\delta}}}\cos \eta^2\ d\eta|\\
    &\geqslant 2\cos1. 
\end{align*}
Since $0<|E|<\infty$, there exists $C_0>0,\ s.t.\ |F|=C_0 N|E|=C_0 N$, then $$\|\tildeG_{\tilde{a}(\cdot)} f_0\|_{L^p(\mathbb{R})}\geqslant(\int_F |2\cos1|^p du)^{\frac{1}{p}}\geqslant C_0 N^{\frac{1}{p}}.$$ If $N$ is chosen large enough, then $\sup\limits_{\{a(\cdot)\}}\|\tildeG_{a(\cdot)} f_0\|_{L^2(\mathbb{R})} >L\|f_0\|_{H^s (\mathbb{R})},$ 
as desired. 
This completes the proof of Proposition \ref{prop:global}. 
\end{proof}

\bibliographystyle{ieeetr}

\end{document}